\documentclass{article}
\usepackage{euscript}
\usepackage{amsmath}
\usepackage{graphicx}
\usepackage{amsthm}
\usepackage{amssymb}
\usepackage{epsfig}
\usepackage[all]{xy}
\usepackage{url}
\usepackage{color}
\usepackage{tkz-berge}
\theoremstyle{theorem}
\newtheorem{theorem}{Theorem}[section]

\newtheorem{lemma}[theorem]{Lemma}

\newtheorem{definition}{Definition}[section]

\numberwithin{equation}{section}
\title{On linear equations arising in Combinatorics (Part III)}
\author{Masood Aryapoor\footnote{E-mail: aryapoor2002@yahoo.com}}
\date{}

\begin{document}

\maketitle




\begin{section}{Introduction}

In the first two papers \cite{Ar1,Ar2} the author embarked on a study of classes of linear equations over integers satisfying a ''Farkas-type" property. 
As the third paper in this study, the present paper deals with another class of linear 
equations over integers that has a similar "Farkas-type" property. Furthermore it is shown that if an arbitrary system of equations over integers satisfies the conditions imposed by Farkas' lemma then it has 
rational solutions of a special type (Theorem \ref{near integers}).

\end{section}
 

\begin{section}{Class $\mathcal{E}_{n}$}

In the first paper \cite{Ar1}, it is shown that if a system of linear equations has a suitable property then the existence of an integral solution 
is decided by a certain set of inequalities (Theorem 3.1 in \cite{Ar1}). In this part, a similar result is presented for another class of linear equations over integers. 
 
\begin{subsection}{Preliminaries}

Let $v=(v_1,...,v_n)\in \mathbb{Q}^n$ and let $s$ be the number of nonzero components of $v$. We want to define  a linear map $L_v:\mathbb{Q}^n\to \mathbb{Q}^{{s \choose 2}+n-s}$ depending on $v$. 
To present a notationally simpler definition, we assume
that $v_{s+1}=...=v_{n}=0$. The linear map $L_v:\mathbb{Q}^n\to \mathbb{Q}^{{s \choose 2}+n-s}$ is defined by the following rule 
\begin{equation} \label{linear map} L_v(t_1,...,t_n)=(\frac{t_1}{v_1}-\frac{t_2}{v_2},...,\frac{t_i}{v_i}-\frac{t_j}{v_j},...,\frac{t_{s-1}}{v_{s-1}}-\frac{t_{s}}{v_{s}},t_{s+1},...,t_n)\end{equation}

Given two elements $v=(v_1,...,v_n),w=(w_1,...,w_n)\in \mathbb{Z}^n$, we write $v|w$ if $w_i$ is divisible by $v_i$ whenever $v_i\neq 0$.
Let $A$ be an abelian subgroup of $\mathbb{Z}^n$ and let $v\in A$ be an element of $A$ such that $v|w$ for all $w\in A$ . Then it is easy to see that 
$L_v(A)\subset \mathbb{Z}^{{s \choose 2}+n-s}$ is a subgroup of $\mathbb{Z}^{{s \choose 2}+n-s}$ and $L_v:A\to\mathbb{Z}^{{s \choose 2}+n-s}$ defines a  $\mathbb{Z}$-linear map. Furthermore 
 the kernel of $L_v$ is $\mathbb{Z} v$ which in particular implies that the rank of $L_v(A)$ (as an abelian group) is equal to (rank of $A$)$-1$.

We inductively define the notion of a mod-linear function $l:\mathbb{Z}^n\to \mathbb{Z}$ of order $\leq r$ where $r$ is a nonnegative integer.
 A mod-linear function of order $\leq 0$ is just one of the projection maps $P_i:\mathbb{Z}^n\to \mathbb{Z}$, $P_i(x_1,...,x_n)=x_i$. When $r>0$, a function $l:\mathbb{Z}^n\to \mathbb{Z}$ is called a mod-linear function 
of order $\leq r$ if there exist mod-linear functions $l_1,l_2:\mathbb{Z}^n\to \mathbb{Z}$ of order $\leq r-1$ and nonzero integers $m_1,m_2$ such that 
$$l(x_1,...,x_n)=\lfloor \frac{l_1(x_1,...,x_n)}{m_1} \rfloor-\lceil \frac{l_2(x_1,...,x_n)}{m_2} \rceil$$
for every $(x_1,...,x_n)\in \mathbb{Z}$. Here, the notations  $\lfloor x \rfloor$ and $\lceil x \rceil$ denote the floor and ceiling functions respectively.

Finally we define inductively a subset $\mathcal{E}_{n}$ of the set of abelian subgroups of $\mathbb{Z}^n$ as follows.   
A nonzero abelian group $A\subset \mathbb{Z}^n$ belongs to $\mathcal{E}_{n}$ if and only if  
there exists a nonzero vector  $v=(v_1,...,v_n)\in A$ satisfying the following properties: (1) $v|w$ for all $w\in A$, and
(2) $L_v(A)=\{0\}$ or $L_v(A)\in \mathcal{E}_{{s \choose 2}+n-s}$ where $L_v$ is defined via \ref{linear map}.

\end{subsection}


\begin{subsection}{A Farkas-type result for $\mathcal{E}_{n}$}  

The following theorem can be considered as a generalization of Theorem 3.1 in \cite{Ar1}.

\begin{theorem}\label{inequalities E}

For every $A\in \mathcal{E}_{n}$ of rank $r$, there exists a finite set $E$, consisting of mod-linear functions $l:\mathbb{Z}^{2n}\to \mathbb{Z}$ of order $\leq r$, for which the following statement holds: For arbitrary integers $a_1\leq b_1,...,a_n\leq b_n$, 
there exists an element $(x_1,...,x_n)\in A$ such that $a_1\leq x_1\leq b_1,...,a_n\leq x_n\leq b_n$ if and only if for every  $l\in E$
we have $0\leq l(a_1,b_1,...,a_n,b_n)$.

\end{theorem}

\begin{proof}

This is proved by induction on $r$. First suppose $r=1$. 
Then there exists an element $v=(v_1,...,v_n)\in A$ such that $A=\mathbb{Z} v$. 
Without loss of generality, we may assume that $v_1,...,v_q>0$, $v_{q+1},...,v_s<0$ and $v_{s+1}=\cdots=v_{n}=0$.
It is obvious that there exists an element $(x_1,...,x_n)\in A$ such that $a_1\leq x_1\leq b_1,...,a_n\leq x_n\leq b_n$, if and only if there exists an integer
$t$ such that $a_1\leq t v_1\leq b_1,...,a_n\leq t v_n\leq b_n$, or equivalently
$$\frac{a_1}{v_1}\leq t \leq \frac{b_1}{v_1},...,\frac{a_q}{v_q}\leq t \leq \frac{b_q}{v_q},$$
$$\frac{b_{q+1}}{v_{q+1}}\leq t \leq \frac{a_{q+1}}{v_{q+1}},...,\frac{b_{s}}{v_{s}}\leq t \leq \frac{a_{s}}{v_{s}},$$
$$a_{s+1}\leq 0\leq b_{s+1},...,a_{n}\leq 0\leq b_{n}.$$
It is easy to see that these inequalities have a common solution $t\in \mathbb{Z}$ if and only if the following conditions hold

$$ 0\leq \lfloor \frac{b_j}{v_j} \rfloor - \lceil \frac{a_i}{v_i} \rceil \quad  \text{when} \quad 1\leq i, j\leq q,$$
$$ 0\leq \lfloor \frac{a_j}{v_j} \rfloor - \lceil \frac{b_i}{v_i} \rceil\quad  \text{when} \quad q< i, j\leq s,$$
$$ 0\leq \lfloor \frac{a_j}{v_j} \rfloor - \lceil \frac{a_i}{v_i} \rceil\quad  \text{when} \quad 1\leq i\leq q<j\leq s,$$
$$ 0\leq \lfloor  \frac{b_i}{v_i} \rfloor- \lceil \frac{b_j}{v_j} \rceil\quad  \text{when} \quad 1\leq i\leq q<j\leq s,$$
$$ 0\leq b_{s+1}-a_{s+1},...,0\leq b_{n}-a_{n}.$$
Using these inequalities, one can easily construct a desired set $E$ for $A$.   

Now suppose $r>1$. Since $A\in \mathcal{E}_{n}$, there exists a nonzero 
element $v=(v_1,...,v_n)\in A$ satisfying the following properties: (1) $v|w$ for every $w\in A$, and
(2) $L_v(A)\in \mathcal{E}_{{s \choose 2}+n-s}$, where $s$ is the number of nonzero components of $v$. 
I claim  that there exists a subgroup $B$ of $A$ such that
$A=B\oplus \mathbb{Z} v$. It is known that such a subgroup $B$ exists if and only if the ableian group $A/\mathbb{Z} v$ is
torsion-free, i.e. if $m w\in \mathbb{Z} v $ for a nonzero element $w\in A$ and a nonzero integer $m$ then $w\in \mathbb{Z} v$. Suppose such an element $w=(w_1,...,w_n)$ and an integer $m$ exist. 
There is nothing to prove if $w=0$. So let $w\neq 0$. There exists a nonzero integer $b$ such that $m w=b v$. 
Since $m\neq 0$, we see that $w_i=0$ if and only if $v_i=0$ for all $i=1,...,n$. From $v|w$, it follows that for all $i$ such that $v_i\neq 0$, we have $w_i=b_i v_i$ where $b_i$ is an integer. 
Therefore we have $m b_i v_i= m w_i =b v_i$, implying $m b_i=b$. It follows that $b$ is divisible by $m$ and consequently $w=\frac{b}{m}v\in \mathbb{Z} v$. The proof  of the claim is complete.

Without loss of generality, we may assume that $v_1,...,v_{q}>0$, $v_{q+1},...,v_{s}<0$ and $v_{s+1}=\cdots=v_{n}=0$. 
Every $x\in A$ can be written as $x=t v+y$ where $t\in \mathbb{Z}$ and $y\in B$. 
It follows that there exists an element $x=(x_1,...,x_n)\in A$ such that $a_1\leq x_1\leq b_1,...,a_n\leq x_n\leq b_n$, if and only if there exist an integer 
$t$ and an element $y=(y_1,...,y_n)\in B$ such that $a_1\leq t v_1+y_1\leq b_1,...,a_n\leq t v_n+y_n\leq b_n$, or equivalently
$$\frac{a_1-y_1}{v_1}\leq t \leq \frac{b_1-y_1}{v_1},...,\frac{a_{q}-y_{q}}{v_{q}}\leq t \leq \frac{b_{q}-y_{q}}{v_{q}},$$
$$\frac{b_{q+1}-y_{q+1}}{v_{q+1}}\leq t \leq \frac{a_{q+1}-y_{q+1}}{v_{q+1}},...,\frac{b_{s}-y_{s}}{v_{s}}\leq t \leq \frac{a_{s}-y_{s}}{v_{s}},$$
$$a_{s+1}\leq y_{s+1}\leq b_{s+1},...,a_{n}\leq y_n\leq b_{n}.$$
One can easily see that these equations  have a common solution $t\in \mathbb{Z}$ if and only if  the following inequalities hold
$$ 0\leq \lfloor \frac{b_j-y_j}{v_j} \rfloor - \lceil \frac{a_i-y_i}{v_i} \rceil \quad  \text{when} \quad 1\leq i, j\leq q,$$
$$ 0\leq \lfloor \frac{a_j-y_j}{v_j}\rfloor - \lceil \frac{b_i-y_i}{v_i} \rceil\quad  \text{when} \quad q< i, j\leq s,$$
$$ 0\leq \lfloor \frac{a_j-y_j}{v_j} \rfloor - \lceil \frac{a_i-y_i}{v_i} \rceil\quad  \text{when} \quad 1\leq i\leq q<j\leq s,$$
$$ 0\leq \lfloor \frac{b_i-y_i}{v_i}  \rfloor- \lceil\frac{b_j-y_j}{v_j}  \rceil\quad  \text{when} \quad 1\leq i\leq q<j\leq s,$$
$$a_{s+1}\leq y_{s+1}\leq b_{s+1},...,a_{n}\leq y_n\leq b_{n}.$$
Using the fact that $v|w$ for all $w\in A$, one can show that these conditions are equivalent to  
the following conditions: \\
(1) For all $1\leq i<j\leq q$, we have
$$ \lceil \frac{a_i}{v_i} \rceil - \lfloor \frac{b_j}{v_j} \rfloor \leq \frac{y_i}{v_i}-\frac{y_j}{v_j}\leq \lfloor\frac{b_i}{v_i} \rfloor - \lceil \frac{a_j}{v_j} \rceil $$
($1'$) For all $1\leq i \leq q$, we have
$$0\leq \lfloor \frac{b_i}{v_i} \rfloor - \lceil \frac{a_i}{v_i} \rceil$$
(2) For all $q< i<j\leq s$, we have
$$ \lceil \frac{b_i}{v_i} \rceil - \lfloor \frac{a_j}{v_j} \rfloor \leq \frac{y_i}{v_i}-\frac{y_j}{v_j}\leq  \lfloor \frac{a_i}{v_i}\rfloor - \lceil\frac{b_j}{v_j}  \rceil$$
($2'$) For all $q< j\leq s$, we have
$$0\leq \lfloor \frac{a_j}{v_j} \rfloor - \lceil \frac{b_j}{v_j} \rceil$$
(3) For all $1\leq i\leq q<j\leq s$, we have
$$  \lceil \frac{a_i}{v_i} \rceil - \lfloor \frac{a_j}{v_j} \rfloor \leq \frac{y_i}{v_i}-\frac{y_j}{v_j}\leq \lfloor\frac{b_i}{v_i}  \rfloor- \lceil  \frac{b_j}{v_j} \rceil $$
(4) $$ a_{s+1}\leq y_{s+1}\leq b_{s+1},...,a_{n}\leq y_n\leq b_{n}.$$
Put
\[ a_{i j} = \left\{ 
 \begin{array}{l l}
\lceil \frac{a_i}{v_i} \rceil - \lfloor \frac{b_j}{v_j} \rfloor  & \quad \text{when} \quad 1\leq i< j\leq q\\
\lceil \frac{b_i}{v_i} \rceil - \lfloor \frac{a_j}{v_j} \rfloor & \quad \text{when} \quad q< i< j\leq s\\
\lceil \frac{a_i}{v_i} \rceil - \lfloor \frac{a_j}{v_j} \rfloor & \quad \text{when} \quad 1\leq i\leq q<j\leq s\\
 \end{array} \right.\]   
and
\[ b_{i j} = \left\{ 
 \begin{array}{l l}
\lfloor\frac{b_i}{v_i} \rfloor - \lceil \frac{a_j}{v_j} \rceil & \quad \text{when} \quad 1\leq i< j\leq q\\
\lfloor \frac{a_i}{v_i}\rfloor - \lceil\frac{b_j}{v_j}  \rceil& \quad \text{when} \quad q< i< j\leq s\\
\lfloor\frac{b_i}{v_i}  \rfloor- \lceil  \frac{b_j}{v_j} \rceil  & \quad \text{when} \quad 1\leq i\leq q<j\leq s.\\
 \end{array} \right.\]   
To complete the proof we need the following lemma. 

\begin{lemma}\label{dimension reduction}

There exists an element $x=(x_1,...,x_n)\in A$ such that $a_1\leq x_1\leq b_1,...,a_n\leq x_n\leq b_n$ if and only if 
for all $1\leq i \leq q$ we have $0\leq \lfloor \frac{b_i}{v_i} \rfloor - \lceil \frac{a_i}{v_i} \rceil$,  for all $q< j\leq s$ we have
$0\leq \lfloor \frac{a_j}{v_j} \rfloor - \lceil \frac{b_j}{v_j} \rceil$, and 
there exists an element  
$(z_{12},...,z_{(s-1)s},z_{1},...,z_{n-s})\in L_{v}(A)$ such that $a_{i j}\leq z_{i j}\leq b_{i j}$ for all $1\leq i<j\leq s$ and $a_{s+1} \leq z_{1}\leq b_{s+1},...,a_n\leq z_{n-s}\leq b_n$.

\end{lemma} 

\begin{proof}

First suppose there exists an element $x=(x_1,...,x_n)\in A$ such that $a_1\leq x_1\leq b_1,...,a_n\leq x_n\leq b_n$. Since $A=B\oplus \mathbb{Z}v$, there exist an integer 
$t$ and an element  $(y_1,...,y_n)\in B$ such that $x_1=y_1+ t v_1,...,x_n=y_n+ t v_n$. As shown above, it follows that $y_1,...,y_n$ satisfy 
Inequalities (1), ($1'$), (2), ($2'$), (3), and (4) above, which in particular implies that the vector 
$L_v(x)\in L_{v}(A)$ satisfies the desired conditions in the lemma and we are done.  

Conversely, suppose there exists an element 
$(z_{12},...,z_{(s-1)s},z_{1},...,z_{n-s})\in L_{v}(A)$ such that $a_{i j}\leq z_{i j}\leq b_{i j}$ for all $1\leq i<j\leq s$ and $a_{s+1} \leq z_{1}\leq b_{s+1},...,a_n\leq z_{n-s}\leq b_n$.
Since $L_v(B)=L_v(A)$, there exists an element $(y_1,...,y_n)\in B$  such that  $L_v(y_1,...,y_n)=(z_{12},...,z_{(s-1)s},z_{1},...,z_{n-s})$. It follows that $y_1,...,y_n$ 
satisfy Inequalities (1), (2), (3), and (4) above. Furthermore, by assumption, Inequalities ($1'$) and ($2'$) hold. As shown above, it follows that there exists  
an element $(x_1,...,x_n)\in A$ such that $a_1\leq x_1\leq b_1,...,a_n\leq x_n\leq b_n$. 

\end{proof}

The group $L_{v}(A)$ has a smaller rank than the group $A$ and $L_{v}(A)\in \mathcal{E}_{{s \choose 2}+n-s}$. By induction, there exists  a finite set $E'$ of mod-linear functions for $A$ which satisfy the corresponding conditions. Since $(-a_{i j})$'s  and $b_{i j}$'s are mod-linear functions (of order $\leq 1$) in terms of $a_i$'s and $b_j$'s 
one can easily show that each element of $E'$ gives rise to a mod-linear function (of order $\leq r$) in terms of $a_i$'s and $b_j$'s. Let $L(E')$ be the set of 
such mod-linear functions in terms of $a_i$'s and $b_j$'s.  Inequalities ($1'$) and ($2'$) give rise to a finite set $E''$ consisting of mod-linear functions of order $\leq 1$. Using Lemma \ref{dimension reduction}, 
one can  easily see that the set $E=L(E')\cup E''$ satisfies the desired condition in the lemma and therefore the proof is complete. 

\end{proof}

\end{subsection}


\end{section}


\begin{section}{Rational solutions of special types}

In this part, it is shown that if a system of linear equations over integers has a rational solution in some interval then it has rational solutions of a particular type in the same interval. 

\begin{definition}

Let  $v_1,...,v_m\in  \mathbb{Q}^n$ be arbitrary vectors. Depending on $v_1,...,v_m$, the set $P_{v_1,...,v_m}$ is defined to be the set of primes $p$ for which there exists 
an integral elementary relation $\sum_{i=1}^m a_i v_i=0$ such that $p|\prod_{a_i\neq 0} a_i$. 

\end{definition}

It is known that there exist only finitely many elementary integral relations among $v_1,...,v_m$ (see \cite{RO}). This implies that the set 
$P_{v_1,...,v_m}$  is a finite (possibly empty) set. Given a set of primes $P$, let $ \mathbb{Q}_P$ denote the following ring 
$$ \mathbb{Q}_P=\{\frac{a}{b}|a,b\in\mathbb{Z}, b\neq 0 \text{ and all prime factors of } b \text{ belong to } P\}.$$

\begin{lemma}\label{k}

Let $v_1,...,v_m\in \mathbb{Q}^n$ be given and put $P=P_{v_1,...,v_m}$. Let  $w\in  \sum_{i= 1}^m \mathbb{Q}_P v_i$ and assume that 
there exists a natural number $k$, with no prime factors in $P$, and a set $I\subset \{1,...,m\}$ such that $k w\in \sum_{i\in I} \mathbb{Q}_P v_i$.  
Then we have $w\in \sum_{i\in I} \mathbb{Q}_P v_i$. 

\end{lemma}

\begin{proof}
  
The proof is by induction on $m-|I|$. There is nothing to prove  when $m-|I|=0$, so suppose $m-|I|>0$. 
Without loss of generality, we may assume that $m\notin I$. By induction on $m-|I|$, we have
$w\in  \sum_{i\in I\cup \{m\}} \mathbb{Q}_P v_i$, i.e.
$w=\sum_{i\in I\cup \{m\}} b_i v_i$, where $b_i\in  \mathbb{Q}_P$ ($i\in I\cup \{m\}$). If $b_m= 0$, then
we are done. So suppose $b_m\neq 0$. We have $k b_m v_m\in \sum_{i\in I}\mathbb{Q}_P v_i$. 
There exists a nonempty set $J\subset I$, such that the vectors $\{v_j\}_{j\in J}$ are linearly independent
and  $k b_m v_m\in  \sum_{i\in J}\mathbb{Q}_p v_i$. It follows that there exists an elementary integral relation  
$\sum_{i\in J\cup \{m\}} a_i v_i=0$. Since  
the vectors $\{v_j\}_{j\in J}$ are linearly independent, we have $a_m\neq 0$. 
Moreover $k$ and $a_m$ are relatively prime, by virtue of the assumption on $k$. Since  $k (b_m v_m), a_m (b_m v_m)\in \sum_{i\in I} \mathbb{Q}_P v_i$,  
we deduce that $b_m v_m\in \sum_{i\in I} \mathbb{Q}_P v_i$ which implies that 
$w\in \sum_{i\in I} \mathbb{Q}_P v_i$ because $w=\sum_{i\in I\cup \{m\}} b_i v_i$.

\end{proof}

\begin{theorem}

Let  $v_1,...,v_m\in \mathbb{Q}^n$ and $a_1\leq b_1,...,a_m\leq b_m$ be in $ \mathbb{Q}_P$ where $P=P_{v_1,...,v_m}$.
If a vector $w\in  \sum_{i=1}^m  \mathbb{Q}_P v_i$ can be written as
$w=\sum_{i=1}^m x_i v_i$ where  $a_1\leq  x_1\leq b_1,...,a_m\leq x_m\leq b_m$ are rational numbers, then
there exist numbers  $a_1\leq  y_1\leq b_1,...,a_m\leq y_m\leq b_m$ in $ \mathbb{Q}_P$ such that $w=\sum_{i=1}^m y_i v_i$.

\end{theorem}

\begin{proof}
 
In the case $P=\emptyset$, the theorem is proved in \cite{Ar1} (Theorem 3.3). In what follows we assume that $P\neq \emptyset$. 
  The proof is by induction on $m$.
First let $m=1$. Since $w\in \mathbb{Q}_P v_1$, we have $w=l v_1$ where  $l\in \mathbb{Q}_P$, implying that 
$w = x_1 v_1=l v_1$. If $v_1=0$, then $y_1=a_1$ satisfies the condition. If  $v_1\neq 0$, then $x_1=l$ and we are done. 
  
Now let $m>1$. If the vectors $v_1,...,v_m$ are  linearly independent, then from 
$w=\sum_{i=1}^m x_i v_i$ and $w\in  \sum_{i= 1}^m  \mathbb{Q}_P v_i$,  
it follows  that $x_1,...,x_m\in \mathbb{Q}_P $ and we are done. So we may assume that $v_1,...,v_m$ are  $\mathbb{Z}$-linearly dependent.
It is easy to see that there exists  a natural number $k$ with no prime factors in  
$P=P_{v_1,...,v_m}$ such that each $x_i$ can be written as $x_i=\frac{N_i}{k}$ where $N_i\in  \mathbb{Q}_P$.  
 We consider two cases.\\
\textit{Case 1}: Assume that there exists a coefficient, say $x_1$, which belongs to $\mathbb{Q}_P $. 
From $k(w-x_1 v_1)=\sum_{i=2}^m N_i v_i$ and Lemma \ref{k}, it follows that $w-x_1v_1\in \sum_{i= 2}^m\mathbb{Q}_P v_i$.
Set $P'=P_{v_2,...,v_m}$. It is clear that $P'\subset P$ and $ \mathbb{Q}_{P'}\subset \mathbb{Q}_P$ using which one can easily show that 
 there exists a natural number $M$ whose prime factors belong to $P\setminus P'$, such that $M (w-x_1v_1)\in \sum_{i= 2}^m\mathbb{Q}_{P'} v_i$. Considering 
the relation $M(w-x_1 v_1)=\sum_{i=2}^m (M x_i) v_i$, we see  that by induction there exists numbers  
 $M a_2\leq  y'_2\leq M b_2,...,M a_m\leq y'_m\leq M b_m$, all in $ \mathbb{Q}_{P'}$, such that 
 $M(w-x_1 v_1)=\sum_{i=2}^m y'_i v_i$.  
The presentation $w=x_1v_1+\sum_{i=2}^m \frac{y'_i}{M} v_i$ satisfies the desired conditions and we are done. \\
\textit{Case 2}: Assume that none of the coefficients $x_1,...,x_m$ belong to $\mathbb{Q}_P $.
In particular, we have $k a_i<N_i<k b_i$ for every $i=1,...,m$. 
Since  $v_1,...,v_m$ are  linearly dependent, there exists an elementary integral relation $\sum_{i=1}^m c_i v_i=0$. 
Without loss of generality, we may assume $c_1\neq 0$. One can easily prove that $\mathbb{Q}_P $ is dense in $\mathbb{R}$ when $P\neq \emptyset$.
Since $\mathbb{Q}_P$ is dense in $\mathbb{R}$ and $k a_i<N_i<k b_i$ for all $i=1,...,m$, one is able to find a rational number 
$r$ such that $k a_1\leq N_1+r c_1=k y_1\leq k b_1$, where $y_1\in \mathbb{Q}_P$ and  $k a_i\leq N_i+r c_i \leq k b_i$ for all $i=2,...,m$. Note that since 
$c_1$ is invertible in  $\mathbb{Q}_P$, we have $r\in  \mathbb{Q}_P$ which implies that $N_i+r c_i\in \mathbb{Q}_P$ for all  $i=1,2,...,m$. Now we have
$w=\sum_{i=1}^m \frac{N_i+r c_i}{k} v_i$ where $\frac{N_1+r c_1}{k}\in \mathbb{Q}_P$ and $k a_i\leq N_i+r c_i\leq k b_i$ for all $i=1,2,...,m$.
We can now use Case 1 to complete the proof.

\end{proof}

Using Farkas' lemma over $\mathbb{Q}$ (Theorem 2.4 in \cite{Ar1}), one can easily derive the following result.  

\begin{theorem}  \label{near integers}

Let  $v_1,...,v_m\in \mathbb{Q}^n$ and $a_1\leq b_1,...,a_m\leq b_m$ be in $ \mathbb{Q}_P$ where $P=P_{v_1,...,v_m}$.
Then a vector $w\in  \mathbb{Q}^n$ can be written as 
$w=\sum_{i=1}^m x_i v_i$ where $a_1\leq x_1\leq  b_1,...,a_m\leq x_m\leq b_m$  belong to $ \mathbb{Q}_P$  if and only if 
$w\in  \sum_{i=1}^m  \mathbb{Q}_P v_i$ and
$$(u,w)\leq \sum_{i=1}^m a_i\frac{(u,v_i)-|(u,v_i)|}{2}+\sum_{i=1}^m b_i\frac{(u,v_i)+|(u,v_i)|}{2},$$
for   every  $\{v_1,...,v_m\}$--indecomposable point $[u]\in \mathbb{R}\mathbb{P}_+^{n-1} $.

\end{theorem}

\end{section}

\end{document}